\documentclass[12pt]{article}

\usepackage{amsmath,amssymb,latexsym,amsthm,enumerate}
\usepackage[top=1.5in, bottom=1.5in, left=1in, right=1in]{geometry}
\usepackage{amsgen}
\usepackage{amstext}
\usepackage{amsbsy}
\usepackage{amsopn}
\usepackage{amsfonts}
\usepackage{eepic}
\usepackage{graphicx}
\usepackage{epsf}
\usepackage{pstricks}

\usepackage{xypic}
\usepackage{xy}
\numberwithin{equation}{section}

\newtheorem{thm}{Theorem}

\newtheorem{theorem}{Theorem}[section]
\usepackage{graphicx}
\newtheorem{lemma}[theorem]{Lemma}
\newtheorem{proposition}[theorem]{Proposition}
\newtheorem{corollary}[theorem]{Corollary}

\newtheorem{defn}[theorem]{Definition}

\def\lcm{{\rm lcm}}

\begin{document}
\title{The pro-supersolvable topology on a free group: deciding denseness}
\author{Claude Marion, Pedro V. Silva, Gareth Tracey}
\maketitle

\begin{center}\small

\bigskip

\end{center}

{\textbf{Abstract.}}
\noindent  Let $F$ be a free group of arbitrary rank and let $H$ be a finitely generated subgroup of $F$. Given a pseudovariety $\mathbf{V}$ of finite groups, i.e. a class of finite groups closed under taking subgroups, quotients and finitary direct products, we endow $F$ with its pro-$\mathbf{V}$ topology. Our main result states that it is decidable whether $H$ is pro-$\mathbf{Su}$ dense, where $\mathbf{Su}\subset \mathbf{S}$  denote respectively the pseudovarieties of all finite supersolvable groups and all finite solvable groups. Our motivation stems from the following  open problem: is it decidable whether $H$ is pro-$\mathbf{S}$ dense? 

\section{Introduction}\label{intro}

The use of topology to deduce group theoretic properties of free groups has been a hugely fruitful endeavour. In Marshall Hall’s seminal paper \cite{Hall50}, a number of natural topologies on a group $G$ are introduced, and used to prove a series of interesting results.

Let $\mathbf{V}$ be a pseudovariety of finite groups and let $G$ be any group. Then $G$ can be endowed with the pro-$\mathbf{V}$ topology. In other words, $G$ is a topological group where the normal subgroups $K$ of $G$ with $G/K \in \mathbf{V}$ form a basis of neighbourhoods of the identity. Hall's work suggests that working with such topologies can lead to surprising group theoretic results, particularly when $G$ is a free group. For example, by working with the pseudovariety of all finite groups (referred to in the literature as the profinite topology), Hall proves that the property of a subgroup of a free group $F$ being the intersection of subgroups of finite index is a topological property, namely being closed in the profinite topology \cite[Theorem 3.4]{Hall50}.

For this reason, the past 70 years or so have seen considerable effort poured into five key questions that have arisen from Hall's work. These questions, for a subgroup $H$ of a free group $F=F(X)$, are as follows:\\ 

\noindent\textbf{Question 1.} Does there exist an algorithm  deciding whether or not  $H$ is $\mathbf{V}$-dense?\\

\noindent \textbf{Question 2.} Does there exist an algorithm deciding whether or not  $H$ is $\mathbf{V}$-closed?\\

\noindent\textbf{Question 3.} Is it decidable whether or not  the pro-$\mathbf{V}$ closure  $\mathrm{Cl}_\mathbf{V}^{F}(H)$ of $H$ is finitely generated?\\

\noindent\textbf{Question 4.}  Given $w\in  F$, is there an algorithm that decides whether or not $w$ belongs to  $\mathrm{Cl}_\mathbf{V}(H)$? \\

\noindent\textbf{Question 5.} Suppose $\mathrm{Cl}_\mathbf{V}^{F}(H)$ is finitely generated. Does  there exist an algorithm that computes a basis (in terms of $X$) for  $\mathrm{Cl}_\mathbf{V}^{F}(H)$?\\

Here, the subgroup $H$ is said to be $\mathbf{V}$-$\mathcal{P}$, for a topological property $\mathcal{P}$, if $H$ has property $\mathcal{P}$ in the pro-$\mathbf{V}$ topology.

When considering Questions 1-5, if $F$ is residually $\mathbf{V}$, i.e. $\mathrm{Cl}_\mathbf{V}^{F}(1)=1$, then  by \cite[Corollary 2.4]{MSTcyclic} one can reduce to the case where $F$ is of finite rank. In particular this holds for  $\mathbf{V}$ nontrivial and extension-closed, or  when $\mathbf{V}$ is the pseudovariety of all finite nilpotent groups; or the pseudovariety of all finite supersolvable groups.

Significant progress has been made on these questions in the case where $\mathbf{V}$ is extension-closed (see \cite{RZ,MSW,MSTmeta}). In particular, in this case, Question 3 has a positive answer \cite{RZ}, while Questions 4 and 5 are equivalent \cite{MSW}. Moreover, Questions 4 and 5 have positive answers when $\mathbf{V}$ is either the pseudovariety of all finite $p$-groups for a fixed prime $p$ (which is extension-closed); or the pseudovariety of all finite nilpotent groups (which is not extension-closed). These facts are proved in \cite{RZ} and \cite{MSW}, respectively. We also mention that in \cite{MSTmeta}, it is explicitly proved that Questions 1, 2 and 5 are equivalent when $\mathbf{V}$ is extension-closed.

In \cite{MSTmeta} we show that Questions 1-5 have positive answers when $\mathbf{V}$ is the pseudovariety of all finite abelian groups; or the pseudovariety of all finite metabelian groups (and neither of them is extension-closed).

For applications of these questions to graph theory and monoid theory, see \cite[Section 6]{HL} and \cite[Section 5]{MSW}.

Apart from the fact that they are equivalent, little is known about Questions 1, 2 and 5 when $\mathbf{V}$ is the pseudovariety of all finite solvable groups. Motivated by this, we focus, in this paper, on the subclass $\mathbf{Su}$ of all finite supersolvable groups. Recall that a finite group is supersolvable if every of its chief factors is cyclic (of prime order). In this case, we can answer Question 1 in the affirmative: 

\begin{thm}\label{t:main} 
Let $F$ be a free group of arbitrary rank. Given a finitely generated subgroup $H$ of $F$, it is decidable whether or not $H$ is $\mathbf{Su}$-dense. 
\end{thm}

The layout  is as follows. In Section \ref{s:p}, we fix some notation, collect some useful facts about pro-$\mathbf{V}$ topologies, and give some preliminary group-theoretical results. In Section \ref{s:sys}, we give some results on the set of prime numbers $p$ for which a given system of polynomial equations in several variables over the integers has a common solution modulo $p$. In Section \ref{s:rl}, we prove a reduction lemma due to Martino Garonzi and define $p$-hypersolvable groups (a useful tool in the proof of Theorem \ref{t:main}).  In Section \ref{s:tl}, we prove a technical lemma needed in the proof of Theorem \ref{t:main}. Finally, we prove Theorem \ref{t:main} in Section \ref{s:mainproof}.  

We are grateful to Martino Garonzi for his insight and Lemma \ref{lem:dense}.

\section{Preliminaries}\label{s:p}

We  first fix  some notation used throughout the paper unless otherwise stated. 
 We let  $F$ be a free group of finite rank $d$ and let $H$ be a given finitely generated subgroup of $F$.   We let $X=\{x_1,\ldots,x_d\}$ be a basis for $F$ and $W=\{w_1,\ldots,w_e\}$ be a basis for $H$ where, for each $1\leq i \leq e$, $w_i$ is a reduced  word over $X\cup X^{-1}$.  
 Recall that if $F$ is a free group of  possibly arbitrary rank, an element of $F$ is called a primitive element if it belongs to a basis of $F$.
  Given a pseudovariety $\mathbf{V}$ of finite groups, we endow $F$ with its pro-$\mathbf{V}$ topology.
  
Let $G$ be a  group. We let $d(G)$ denote the minimal number of generators for $G$, $|G|$ denote the order of $G$, and for an element $g\in G$ we let ${\rm ord}(g)$ denote the order of $g$ ($d(G)$, $|G|$ and ${\rm ord}(g)$ are possibly infinite). 
If  $g, h\in G$, we let $[g,h]=ghg^{-1}h^{-1}$.
 We also denote by $\mathrm{Aut}(G)$ the group of automorphisms of $G$. 
 If $H$ is a subgroup of $G$ then the core $\textrm{Core}_G(H)$ of $H$ in $G$ is the largest normal subgroup of $G$ contained in $H$ and is equal to $\bigcap_{u\in G} u^{-1}Hu$.
 
Given a positive integer $a$, we let $\mathbb{Z}/a\mathbb{Z}$ denote a cyclic group of order $a$. When clear from the context, we consider $\mathbb{Z}/a\mathbb{Z}$ as a ring and denote by $(\mathbb{Z}/a\mathbb{Z})^*$ its group of units. 
Moreover given a prime $p$ and a positive integer $n$, we let $V_n:=(\mathbb{Z}/p\mathbb{Z})^n$ denote the elementary abelian $p$-group of order $p^n$ that can also be seen as a vector space over $\mathbb{Z}/p\mathbb{Z}$ of dimension $n$.

We let $\mathbf{Ab}$ denote the pseudovariety of all finite abelian groups. Given  a positive integer $a$,  $\mathbf{Ab}(a)$ is the pseudovariety of  all finite abelian groups of exponent dividing $a$. In particular, given a prime $p$, $\mathbf{Ab}(p)$ is the pseudovariety of all finite elementary abelian $p$-groups.

Given integers $a$ and $b$ not both equal to $0$, $\mathrm{gcd}(a,b)$ denotes the greatest common divisor of $a$ and $b$, while if $a$ and $b$ are nonzero $\textrm{lcm}(a,b)$ denotes the lowest positive common multiple of $a$ and $b$. We naturally extend the notation to tuples of integers. More precisely, if $(a_1,\ldots,a_n)$ is a tuple of integers not all equal to 0, then $\mathrm{gcd}(a_1,\ldots,a_n)$ is the greatest positive integer dividing each of $a_1,\ldots,a_n$, whereas if $(b_1,\ldots,b_n)$ is a tuple of nonzero integers then $\mathrm{lcm}(b_1,\ldots,b_n)$ is smallest positive integer which is a multiple of each of  $b_1,\ldots,b_n$.  

For $k\in \mathbb{N}$ and $1\leq i \leq k$, let $\iota_i \in \mathbb{Z}^{k}$ have $1$ in the $i$th coordinate and 0 everywhere else. Given a prime $p$, we sometimes view $\iota_i$ as an element of $(\mathbb{Z}/p\mathbb{Z})^k$.
We denote by $\mathbb{P}$  the set of all primes.

We now collect some facts on pro-$\mathbf{V}$ topologies. For a pseudovariety $\mathbf{V}$, considering finite groups endowed with the discrete topology, the pro-$\mathbf{V}$ topology on a group $G$ is defined as the coarsest topology which makes all homorphisms from $G$ into elements of $\mathbf{V}$ continuous. Equivalently $G$ is a topological group and the normal subgroups $K$ of $G$ such that $G/K\in \mathbf{V}$ form a basis of neighbourhoods of the identity. 
A pseudovariety $\mathbf{V}$ is extension-closed if whenever $G$ is a finite group with a normal subgroup $N$ such that $N$ and $G/N$ belong to $\mathbf{V}$ then $G$ also belongs to $\mathbf{V}$. The \textit{trivial pseudovariety} consists of all trivial groups. 
A subgroup $H$ of $G$ is pro-$\mathbf{V}$ open if and only if it is pro-$\mathbf{V}$ clopen, and if and only if $G/\textrm{Core}_G(H)$ belongs to $\mathbf{V}$.  Note that $H$ is pro-$\mathbf{V}$ closed if and only if, for every $g\in G\setminus H$, there exists some pro-$\mathbf{V}$ clopen $K\leq G$ such that $H \leq K$ and $g\not \in K$. Moreover a subgroup $H$ of $G$ is pro-$\mathbf{V}$ dense if and only if $HN=G$ for every normal subgroup $N$ of $G$ with $G/N\in \mathbf{V}$.
Given $S \subseteq G$, we also denote by  $\mathrm{Cl}_{\mathbf{V}}^G(S)$ the $\mathbf{V}$-closure of $S$ in $G$. If $H \leq G$, then also $\mathrm{Cl}_{\mathbf{V}}^G(H) \leq G$ \cite[Theorem 3.3]{Hall50}.

Suppose that $\mathbf{V}$ and $\mathbf{W}$ are pseudovarieties of finite groups. Then:
\begin{equation}\label{e:cl}
\mbox{If $\mathbf{W} \subseteq \mathbf{V}$, then $\mathrm{Cl}_{\mathbf{V}}^G(S) \subseteq \mathrm{Cl}_{\mathbf{W}}^G(S)$ for every $S \subseteq G$.}
\end{equation}
This follows from the pro-$\mathbf{W}$ topology being coarser than the pro-$\mathbf{V}$ topology on $G$.

We will also use the following result.
 
\begin{lemma}
\label{autospres}
Let {\bf V} be a pseudovariety of finite groups, let $H$ be a subgroup of a group $G$ and let $\varphi \in \mathrm{Aut}(G)$. If $H$ is {\bf V}-dense in $G$, then so is $\varphi(H)$.
\end{lemma}

\begin{proof}
Assume that $H$ is {\bf V}-dense in $G$. Let $N \unlhd G$ be such that $G/N \in {\bf V}$. Then $\varphi^{-1}(N) \unlhd G$ and $G/\varphi^{-1}(N) \cong G/N \in {\bf V}$. Since $H$ is {\bf V}-dense in $G$, we have $H\varphi^{-1}(N) = G$ and so 
$$(\varphi(H))N = \varphi(H\varphi^{-1}(N)) = \varphi(G) = G.$$
Therefore $\varphi(H)$ is {\bf V}-dense in $G$.
\end{proof}

In the remainder of this section we give other group-theoretical results.
The result below is well known as Gaschutz's lemma.

\begin{lemma}\label{l:gaschutz} {\rm{\cite{Gaschutz}}}
Let $\pi: G\rightarrow K$ be  a surjective homomorphism of finite groups. Suppose $m\geq d(G)$  and  $(k_1,\ldots,k_m)$ is a generating tuple for $K$. Then there exists a generating tuple 
$(g_1,\ldots,g_m)$ for $G$ with $\pi(g_i)=h_i$ for $1\leq i \leq m$.
\end{lemma}

We apply Gaschutz's lemma in the next result.

\begin{lemma}\label{l:cyclicga}
Let $G$ be a finite cyclic group and let $\{g_1,\ldots,g_m\}$ be a generating set for $G$. Then for $1\leq i <m$ there exist integers $r_i$  such that $g_m\prod_{i=1}^{m-1}g_i^{r_i}$ generates $G$. 
\end{lemma}

\begin{proof}
We prove the result by induction on $m$. If $m=1$ the result is trivial since $g_1$ generates $G$. 
We assume $m>1$. Let $N=\langle g_1\rangle$. Now $G/N$ is cyclic with generating set $\{Ng_2,\ldots,Ng_m\}$. By the induction hypothesis for $1<i<m$ there exist integers $r_i$ such that $G/N$ is generated by 
$Ng_m\prod_{i=2}^{m-1}g_i^{r_i}$. Since $G$ is cyclic, by Lemma \ref{l:gaschutz}, there exists $n=g_1^{r_1}\in N$ where $r_1$ is an integer such that $g_m\prod_{i=1}^{m-1}g_i^{r_i}$ generates $G$. This completes the induction. 
\end{proof}

We can now apply Lemma \ref{l:cyclicga} to prove the following result. 

\begin{lemma}
\label{primcyclic} Let $F$ be a free group of finite rank $d$ over $X=\{x_1,\ldots,x_d\}$ and let $c$ be a positive integer.
Let $\varphi:F \to \mathbb{Z}/c\mathbb{Z}$ be a surjective homomorphism. Then there exists some $\lambda \in \mathrm{Aut}(F)$ such that $\langle \varphi(\lambda(x_1))\rangle = \mathbb{Z}/c\mathbb{Z}$.
\end{lemma}

\begin{proof}
Write $C= \mathbb{Z}/c\mathbb{Z}$. The case $d = 1$ is immediate, hence we assume that $d > 1$. 
For $1\leq i \leq d$, let $c_i=\phi(x_i)\in C$. As $\varphi: F\rightarrow C$ is surjective, $\{c_1,\ldots,c_d\}$ is a generating set for $C$. By Lemma \ref{l:cyclicga} for $1\leq i < d$ there exist integers $r_i$ such that $c_d\prod_{i=1}^{d-1}c_i^{r_i}$ generates $C$. Now since $\{x_1,\ldots,x_d\}$ is a basis of $F$ and $\{x_1,\ldots,x_{d-1},x_d\prod_{i=1}^{d-1}x_i^{r_i}\}$ is clearly a generating set for $F$, the latter set is a basis of $F$ (as $F$ is hopfian). Thus $x_d\prod_{i=1}^{d-1}x_i^{r_i}$ is a primitive word of $F$ and there exists  $\lambda \in \mathrm{Aut}(F)$ such that $\lambda(x_1)=x_d\prod_{i=1}^{d-1}x_i^{r_i}$. Now $$\varphi(\lambda(x_1))=\varphi\left(x_d\prod_{i=1}^{d-1}x_i^{r_i}\right)=\varphi(x_d)\prod_{i=1}^{d-1}\varphi(x_i)^{r_i}=c_d\prod_{i=1}^{d-1}c_i^{r_i}$$ and so  
$\varphi(\lambda(x_1))$ generates $C$. 
\end{proof}

\section{Systems of polynomial equations over the integers}\label{s:sys}

In this section, given a system of polynomial equations in several variables over the integers, we are interested in properties of the set of primes $p$ for which the system has a common solution modulo $p$. We first introduce some notation.  

 \begin{defn}
 Let $J=\{f_1,\ldots,f_m\} \subset  \mathbb{Z}[X_1,\ldots,X_n]$ where  $m,n\geq 1$. We define $S(J)$ to be the set of primes $p$ such that $f_1,\ldots,f_m$ have a common root modulo $p$. 
\end{defn}

The result below establishes that one can essentially reduce the original system of polynomials to a single polynomial in one variable. 

\begin{theorem}\label{t:jar}{\rm{\cite[Theorem 1.2]{Jarviniemi}}}
Let $J=\{f_1,\ldots,f_m\} \subset  \mathbb{Z}[X_1,\ldots,X_n]$ where  $m,n\geq 1$.  There is a polynomial $f\in \mathbb{Z}[X]$ such that $S(J)=S(f)$.
\end{theorem}

Let $J=\{f_1,\ldots,f_m\} \subset \mathbb{Z}[X_1,\ldots,X_n]$ where  $m,n\geq 1$, and let $f\in \mathbb{Z}[X]$ be such that $S(J)=S(f)$. The existence of such a polynomial $f$ is guaranteed by Theorem \ref{t:jar}.  
The aim of this section is to show that one can decide whether $f$ is a nonzero constant polynomial, and moreover whether $f=\pm1$ (see Corollary \ref{c:sys} below).  

 \begin{defn}
 Let $J=\{f_1,\ldots,f_m\} \subset \mathbb{C}[X_1,\ldots,X_n]$ where  $m,n\geq 1$. 
 We let $I=\langle J\rangle$ be the ideal of $\mathbb{C}[X_1,...,X_n]$ generated by $J$ and set $$V(I)=\{(x_1,\ldots,x_n)\in \mathbb{C}^n: g(x_1,\ldots,x_n)=0\  \textrm{for all} \ g \in I\}.$$ We also let $\mathcal{G}$ be the unique reduced Gr\"{o}bner basis of $I$ with respect to a fixed monomial order on $\mathbb{C}[X_1,...,X_n]$.   
\end{defn}

We record below three  results needed  for the proof of Corollary \ref{c:sys} below. 

\begin{proposition}\label{p:dries}{\rm{\cite[Proposition 2.7]{vDries}}}
Let $J=\{f_1,\ldots,f_m\} \subset  \mathbb{Z}[X_1,\ldots,X_n]$ where  $m,n\geq 1$.  Then $S(J)$ is finite if and only if $V(\langle J \rangle)= \emptyset$. 
\end{proposition}

\begin{lemma}\label{l:sjf}
Let $J=\{f_1,\ldots,f_m\} \subset  \mathbb{Z}[X_1,\ldots,X_n]$ where  $m,n\geq 1$. Then it is decidable whether $V(\langle J \rangle)= \emptyset$. Consequently, it is also decidable whether $S(J)$ is finite. 
\end{lemma}

\begin{proof}
Fix a monomial order  on $\mathbb{C}[X_1,...,X_n]$. There is an algorithm to compute the (unique) reduced Gr\"{o}bner basis $\mathcal{G}$ of $\langle J\rangle$ (see \cite{Buchberger}). Since $\langle \mathcal{G}\rangle=\langle J\rangle$, Hilbert's Nullstellensatz implies that  $V(\langle J\rangle)=\emptyset$ if and only if $\mathcal{G}=\{1\}$. Thus it is decidable whether $V(\langle J\rangle)=\emptyset$, and Proposition \ref{p:dries} yields that is decidable whether $S(J)$ is finite.  
\end{proof}

\begin{proposition}\label{p:fc}
Let $f\in \mathbb{Z}[X]$. Then $S(f)$ is finite if and only if $f$ is a nonzero constant polynomial. 
\end{proposition}

\begin{proof}
Clearly if $f$ is a nonzero constant polynomial then $S(f)$ is finite. 
Also if $f=0$ then $S(f)=\mathbb{P}$. 
Suppose now that $f$ is not a constant polynomial. Write $f(X)=\sum_{i=0}^n a_i X^{i}$. If $a_0=0$, then $f(p)\equiv 0 \mod p$ for every prime $p$, and so $S(f)=\mathbb{P}$. 
We can therefore assume that $a_0\neq 0$. For any nonzero $t\in \mathbb{Z}$, let 
\begin{eqnarray*}
g_t(X) & = & f(a_0tX)\\
& = & \sum_{i=0}^n a_i(a_0tX)^{i}\\
& = & a_0 \left[1+\sum_{i=1}^n a_ia_0^{i-1}t^{i}X^{i}\right]
\end{eqnarray*}
and consider 
$$h_t(X)= 1+\sum_{i=1}^n a_ia_0^{i-1}t^{i}X^{i}.$$
Since $h_t(X)$ can take the values $-1$, $0$ and $1$ only at finitely many points, there exists a prime $p$ such that $h_t(n)\equiv 0 \mod p$ for some integer $n$.  Note that $h_t(n)\equiv 1 \mod t$, and so ${\rm gcd}(t,p)=1$.
Also, since $f(a_0tn)=a_0h_t(n)$, then $f(a_0tn)\equiv 0 \mod p$ and $a_0tn$ is a root  of $f$ modulo $p$. 
Starting with $t=1$, we get a prime $p_1$ such that $f$ has a root modulo $p_1$ (and ${\rm gcd}(t,p_1)=1$). Then, setting $t=p_1$, we get a prime $p_2$ such that $f$ has a root modulo $p_2$ and ${\rm gcd}(p_1,p_2)=1$. In particular $p_1$ and $p_2$ are distinct. Suppose  we have obtained a sequence of distinct primes $p_1,\ldots,p_m$ such that $f$ has a root modulo $p_i$ for $1\leq i \leq m$.  Setting  $t=\prod_{i=1}^m p_i$, we get a prime $p_{m+1}$ such that $f$ has a root modulo $p_{m+1}$ and ${\rm gcd}(\prod_{i=1}^m p_i,p_{m+1})=1$. In particular, $p_1,\ldots,p_{m+1}$ are distinct primes. In this way, we get an infinite sequence  $(p_i)$  of primes such that  $f$ has a root modulo $p_i$ for every $i$. 
\end{proof}

We can now derive the result aforementioned above. 

\begin{corollary}\label{c:sys}
Let $J=\{f_1,\ldots,f_m\} \subset  \mathbb{Z}[X_1,\ldots,X_n]$ where  $m,n\geq 1$. Let $f\in \mathbb{Z}[X]$ be such that  $S(J)=S(f)$. The following assertions hold.  
\begin{enumerate}[(i)]
\item It is decidable whether $f$ is a nonzero constant polynomial. 
\item It is decidable whether $f=\pm1$.
\end{enumerate}
\end{corollary}

\begin{proof}
The first part follows from Lemma \ref{l:sjf} and Proposition \ref{p:fc}. 
We now consider the second part. If $f$ is not a nonzero constant polynomial, then $f\neq\pm 1$ and we are done. We can therefore suppose that $f$ is a nonzero constant polynomial and we would like to decide whether $f=\pm1$, i.e whether $S(J)$ is empty. Since $S(J)=S(f)$ is finite, it follows from Proposition \ref{p:dries} that $V(\langle J\rangle)=\emptyset$ and by Hilbert's Nullstellensatz $\langle J\rangle=\mathbb{C}[X_1,\ldots,X_n]$. In particular, there exist polynomials $g_1,\ldots,g_m\in \mathbb{C}[X_1,\ldots,X_n]$ such that $\sum_{i=1}^m f_ig_i=1$. 
The equation $\sum_{i=1}^m f_ig_i=1$ amounts to a system of linear equations with integer coefficients (since the $m$ polynomials $f_i$ are in $\mathbb{Z}[X_1,...,X_n]$). We know that this system has a (complex) solution, so it must also have a rational solution. Scaling up, there exist an integer $a$ and polynomials $\ell_1,\ldots,\ell_m \in \mathbb{Z}[X_1,...,X_n]$ such $\sum_{i=1}^m f_i\ell_i=a$. 

 By \cite[Chapter 4, Theorem IV]{MW}, which gives an Effective  Hilbert's Nullstellensatz,  $a$ and $\ell_1,\ldots,\ell_m$ are computable.  Note that if $a\in \pm\{1\}$ then $S(J)$ is empty and $f=\pm1$. 
 
 Suppose that $a\not \in \{\pm1\}$. Now if $f\neq \pm 1$, that is, if $S(J)$ is nonempty, then every prime  $p\in S(J)$ must divide $a$. So we just need to check whether the polynomials in $J$ have a common root modulo the primes dividing $a$. Then $f\neq \pm 1$ if and only if $S(J)$ is nonempty if and only if the polynomials in $J$ have a common root modulo a prime dividing $a$. This latter condition is decidable.  
\end{proof}

\section{A reduction lemma and $p$-hypersolvable groups}\label{s:rl}

 In this section we prove the reduction lemma due to Martino Garonzi, namely Lemma \ref{lem:dense} below.
Given a prime $p$, we then introduce the concept of a $p$-hypersolvable group and give some of its properties that we will use later.

We first consider the following definition generalising the notion of denseness in pro-$\mathbf{V}$ topologies on a group $G$.

\begin{defn}\label{d:gd}
Let $G$ be a group and let $H$ be a subgroup of $G$. For a class $\mathcal{C}$ of finite groups, let $P_{\mathcal{C}}(H,G)$ be the following property: $HN=G$ whenever $N$ is a normal subgroup of $G$ with $G/N\in \mathcal{C}$. 
\end{defn}

We  characterise below the finite primitive supersolvable groups. Recall that a finite group $L$ is primitive if it has a maximal subgroup $M$ such that ${\rm Core}_L(M)$ is trivial.

\begin{lemma}\label{l:supprim}
A finite group is  primitive supersolvable if and only if it is a semidirect product $\mathbb{Z}/p\mathbb{Z}\rtimes C$ where $p$ is a prime and $C\leq \mathrm{Aut}(\mathbb{Z}/p\mathbb{Z})\cong (\mathbb{Z}/p\mathbb{Z})^*$.
\end{lemma}

\begin{proof}
We use some elementary facts about finite primitive groups, for more details see \cite{Baer}.
Let $G$ be a finite  primitive supersolvable group. Since $G$ is solvable, a minimal normal subgroup $N$ of $G$ is (elementary) abelian. Since $G$ is primitive and $N$ is abelian, $N$ is in fact the unique minimal normal subgroup of $G$ and $C_G(N)=N$. Since $G$ is primitive, $G$ has a core free maximal subgroup $C$. Therefore $CN=G$. Also $C\cap C_G(N)=1$ and so $N\cap C=1$.  In particular $G=N\rtimes C$. Finally, since $G$ is supersolvable and $N$ is a minimal normal subgroup of $G$, we obtain $N\cong \mathbb{Z}/p\mathbb{Z}$ for some prime  $p$.  As $C_G(N)=N$, we obtain $C\leq {\rm Aut}(\mathbb{Z}/p\mathbb{Z})\cong (\mathbb{Z}/p\mathbb{Z})^*$.

Conversely, let $G=\mathbb{Z}/p\mathbb{Z}\rtimes C$ where $p$ is prime and $C\leq \mathrm{Aut}(\mathbb{Z}/p\mathbb{Z})\cong (\mathbb{Z}/p\mathbb{Z})^*$. It is clear that $G$ is supersolvable.  If $C=1$, then $G$ is primitive (with maximal subgroup 1). Suppose that $C\neq 1$. As $G$ is supersolvable and $C$ has prime index in $G$, $C$ is a maximal subgroup of $G$. Moreover $\mathrm{Core}_G(C)=1$, as $\mathbb{Z}/p\mathbb{Z}$ is the unique nontrivial normal subgroup of $G$. Hence $G$ is primitive. 
\end{proof}

We now introduce a notion used in Lemma \ref{lem:dense} below.

\begin{defn}\label{d:cqp}
Let $G$ be a group and let $p$ be a prime. We say that a subgroup $H$ of $G$ \emph{satisfies condition $Q_p(H,G)$} if $HN=G$ for all $N\unlhd G$ with $F/N\cong  \mathbb{Z}/p\mathbb{Z}\rtimes C$ for some  $C\leq {\rm Aut}(\mathbb{Z}/p\mathbb{Z})$. 
\end{defn}

\begin{lemma}\label{lem:dense}
Let $G$ be a nontrivial group and let $H$ be a subgroup of $G$.  Let $\mathcal{C}$ be a nontrivial class of finite groups closed under taking quotients. Let $\mathcal{C}_{\rm prim}=\{L \in \mathcal{C}: L \ \textrm{primitive}\}$.  Let $\mathbf{V}$ be a pseudovariety of finite groups. The following assertions hold.
\begin{enumerate}[(i)]
\item $P_{\mathcal{C}}(H,G)$ holds if and only if $P_{\mathcal{C}_{\rm prim}}(H,G)$ holds. 
\item $H$ is $\mathbf{V}$-dense if and only if $P_{\mathbf{V}_{\rm prim}}(H,G)$ holds. 
\item $H$ is $\mathbf{Su}$-dense if and only if $H$ satisfies condition $Q_p(H,G)$ for every prime $p$. 
\end{enumerate}
\end{lemma}

\begin{proof}
We  consider the first part. It is clear that if $P_{\mathcal{C}}(H,G)$ holds then $P_{\mathcal{C}_{\rm prim}}(H,G)$ also holds. Suppose now that $P_{\mathcal{C}_{\rm prim}}(H,G)$ holds.  Assume for a contradiction that $P_{\mathcal{C}}(H,G)$ does not hold.  Then there exists $N\unlhd G$ with $G/N\in \mathcal{C}$ such that $HN\neq G$. Let $M$ be a maximal subgroup of $G$ containing $HN$ and set $L=\textrm{Core}_G(M)$. Note that $L$ contains $N$.  Since $\mathcal{C}$ is closed under taking quotients and $G/N\in \mathcal{C}$, we obtain that  $G/L\cong \frac{G/N}{L/N}$ belongs to $\mathcal{C}$.  Now $M/L$ is a maximal subgroup of $G/L$ and ${\rm Core}_{G/L}(M/L)=1$, hence $G/L\in \mathcal{C}_{{\rm prim}}$. Since  $P_{\mathcal{C}_{\rm prim}}(H,G)$ holds,  $HL=G$. As $L\leq M$, this implies $HM=G$, a contradiction since $HM=M$. 

The second part is now immediate and the final part follows from Lemma \ref{l:supprim} and Definition \ref{d:cqp}. 
\end{proof}

Given a prime $p$, we now introduce the family of $p$-hypersolvable groups. 

\begin{defn}
Fix a prime $p$. For a positive integer $n$, let $V_n$ denote the elementary abelian group of order $p^n$. A  finite group $G$ is said to be  \emph{$p$-hypersolvable} if $G$ is isomorphic to  $V_n\rtimes C$ for some positive integer $n$ and some  $C\leq \mathrm{Aut}(\mathbb{Z}/p\mathbb{Z})\cong (\mathbb{Z}/p\mathbb{Z})^*$  acting diagonally on $V_n$.  We let $x$ be a generator for $C$ and identify $x$ with an integer $1\le \alpha\le p-1$. In this way, $V_n \rtimes C=V_n \rtimes_{\alpha}C$ via the action $v^x=\alpha v$ for all $v\in V_n$, endowed with the product $(v,x^a)(w,x^b) = (v + w^{(x^{a})},x^{a+b}) = (v + \alpha^{a}w,x^{a+b})$. (Note that $\alpha=1$ if and only if $C=1$.)
\end{defn}

Clearly the $p$-hypersolvable groups are supersolvable. The $p$-hypersolvable groups in fact have a number of other useful properties, which we collect in the next lemma. 
\begin{lemma}\label{lem:facts}
Fix a prime $p$ and a positive integer $n$. Let $G=V_n\rtimes_{\alpha} C$ be a $p$-hypersolvable group. Set $c=|C|$.  Let $g,h\in G$ and write $g,h$ uniquely in the form $g=vx^{c_1}$ and $h=wx^{c_2}$, where $v,w\in V_n$, and $0\le c_1,c_2\le c-1$. Then 
\begin{enumerate}[\upshape(i)]
\item We have $g^{-1}=(-\alpha^{-c_1}v)x^{-c_1}$.
\item We have $[g,h]=(1-\alpha^{c_2})v+(\alpha^{c_1}-1)w$. 
\item If $g\neq 1$ then ${\rm ord}(g)=p$ if $c_1=0$, and ${\rm ord}(g)=c/{\rm gcd}(c,c_1)$ otherwise.
\item If  $c_1\neq 0$ then there exists $h\in G$ such that $hgh^{-1}=x^{c_1}\in C$.
\item If $u \in V_n$ then $gug^{-1}=\alpha^{c_1}u$. 
\end{enumerate}
\end{lemma}

\begin{proof}
Part (i) is clear. Indeed $(-\alpha^{-c_1}v)x^{-c_1}\in G$,
$$g\cdot((-\alpha^{-c_1}v)x^{-c_1})=(vx^{c_1})\cdot((-\alpha^{-c_1}v)x^{-c_1})=(v+\alpha^{c_1}(-\alpha^{-c_1}v))(x^{c_1}x^{-c_1})=1.$$
We now consider part (ii). Using part (i), we have
\begin{eqnarray*}
[g,h] & = & ghg^{-1}h^{-1}\\
& = & (vx^{c_1})\cdot(wx^{c_2})\cdot((-\alpha^{-c_1}v)x^{-c_1})\cdot((-\alpha^{-c_2}w)x^{-c_2})\\
& = &  ((v+\alpha^{c_1}w)x^{c_1+c_2}) \cdot ((-\alpha^{-c_1}v-\alpha^{-c_1-c_2}w)x^{-c_1-c_2})  \\
& =  &(v+\alpha^{c_1}w)+\alpha^{c_1+c_2}(-\alpha^{-c_1}v-\alpha^{-c_1-c_2}w)\\
& =& (1-\alpha^{c_2})v+(\alpha^{c_1}-1)w. 
\end{eqnarray*}

 We consider part (iii). Given a positive integer $i$, arguing by induction yields  $$g^{i}=\sum_{j=0}^{i-1}\alpha^{jc_1}vx^{ic_1}.$$ 
 Suppose that $g\in V_n$, so that $c_1=0$. As $g$ is nontrivial and $V_n$ is an elementary abelian $p$-group, $g$ has order $p$. 
 Suppose now that $g \not \in V$. Then $c_1 \neq 0$ and for $i\geq 1$
 $$g^{i}=\frac{\alpha^{ic_1}-1}{\alpha^{c_1}-1}vx^{ic_1}.$$
  As elements of $(\mathbb{Z}/p\mathbb{Z})^*$, $x$  and $\alpha$ have order $c$, and so $x^{c_1}$ and $\alpha^{c_1}$ have order $c/\mathrm{gcd}(c,c_1)$. It now follows that $g$ has order $c/\mathrm{gcd}(c,c_1)$. 
  We now consider part (iv). Setting $h:=v/(\alpha^{c_1}-1)$, one checks that $hgh^{-1}=x^{c_1}\in C.$
  We finally consider part (v). By part (i), $g^{-1}=(-\alpha^{-c_1}v)x^{-c_1}$ and so
  \begin{eqnarray*}
  gug^{-1} & = & (vx^{c_1})\cdot u \cdot  ((-\alpha^{-c_1}v)x^{-c_1})  \\
  & = &((v+\alpha^{c_1}u)x^{c_1})\cdot  ((-\alpha^{-c_1}v)x^{-c_1})\\
  & = & v+\alpha^{c_1}u+\alpha^{c_1}(-\alpha^{-c_1})v\\
  & = &  \alpha^{c_1}u.
  \end{eqnarray*}
  \end{proof}

We characterise  below generating sets in (non elementary abelian) $p$-hypersolvable groups.

\begin{lemma}\label{lem:genpsup}
Fix a prime $p$ and an integer $d\geq 2$. Let $G:=V_{d-1}\rtimes C$ be an associated $p$-hypersolvable group with $C\neq 1$ and let $c=|C|$. For an element $g\in G$, write $g$ uniquely in the form $g=v(g)x^{c(g)}$, where $v(g)\in V_d-1$ and $0\le c(g)\le c-1$.
Then a subset $Z:=\{g_1,\hdots,g_d\}$ of $G$ of cardinality $d$ is a generating set for $G$ if and only if each of the following holds:
\begin{enumerate}[\upshape(i)]
    \item $\lcm\left\{\dfrac{c}{\mathrm{gcd}(c,c(g_i))}\text{ : }1\leq i \leq d, \ c(g_i)\neq 0\right\}=c$.
    \item $\{[g_i,g_j]\text{ : }g_i,g_j\in Z\}$ spans $V_d-1$.
\end{enumerate}
\end{lemma}
\begin{proof}
If (i) and (ii) hold, then it is clear that $Z$ generates $G$.
So assume that $Z$ is a generating set for $G$. Then the reduction of $Z$ modulo $V_{d-1}$ is a generating set for the cyclic group $C$. It follows that $Z$ satisfies (i).
Next, we prove (ii). Since $V_{d-1}=[G,G]$ and $Z$ generates $G$, it follows easily from the commutator identities and induction that $V_{d-1}$ is generated by the normal closure of the set $\{[g_i,g_j]\text{ : }g_i, g_j\in Z\}$. 
Since every subgroup of $V_{d-1}$ is normal in $G$, part (ii) follows.
\end{proof}

Finally, using Gaschutz's lemma, we determine the minimal number of generators of a $p$-hypersolvable group. 

\begin{lemma}
\label{diameter}
Fix a prime $p$ and a positive integer $n$.  Let $G:=V_{n}\rtimes C$ be an associated $p$-hypersolvable group.
Then
$$d(G) = \left\{
\begin{array}{ll}
n&\mbox{ if $C=1$}\\
n+1&\mbox{ otherwise.}
\end{array}
\right.$$
\end{lemma}

\begin{proof}
The case $C=1$ is well known, hence we assume that $C\neq 1$. We have $C \cong \mathbb{Z}/c\mathbb{Z}$ for some  divisor $c>1$ of $p-1$.  Let $x$ be a generator of $C$.  It is clear that $d(G)\leq n+1$. We prove that $d(G)=n+1$ by induction on $n$. 
 The case $n=1$ is trivial since $V_1\rtimes C$ is not cyclic. 
Suppose $n>1$. Let $\langle v_1,\ldots, v_n\rangle$ be a basis of $V_n$ and let $U=\langle v_n\rangle$. Note that $U\lhd G$.  By the induction hypothesis, we have $d(G/U)=n$. Thus since 
$d(G/U)\leq d(G)\leq n+1$, we may assume (in pursuit of a contradiction) that $d(G)=n=d(G/U)$. Since $\{Uv_1,\ldots, Uv_{n-1},Ux\}$ is a generating set of $G/U$, Lemma \ref{l:gaschutz} yields that for $1\leq i\leq n$, there exist $u_i\in U$  such that $G=\langle u_1v_1,\ldots, u_{n-1}v_{n-1},u_nx \rangle$.  Let $N=\langle u_iv_i: 1\leq i \leq n-1\rangle$. Then $N\leq V_n$ and so $N\lhd G$. Since $u_nx$ has order $c$, we obtain 
$$G=N\rtimes \langle u_nx\rangle.$$ This is a contradiction, as $|G|=p^nc$ but  $N\rtimes \langle u_nx\rangle$ has order at  most $p^{n-1}c$. Hence $d(G)=n+1$. 
\end{proof}

\section{A technical lemma}\label{s:tl}
The aim of this section is to prove the technical lemma below. Let $F$ be a free group of finite rank $d>1$ over $X=\{x_1,\ldots,x_d\}$.
For $i = 1,\ldots,d-1$, let $\tau_i \in S_d$ denote the transposition $(i\; d)$, and let $\tau_d$ be the identity. We also denote by $\tau_i$ the automorphism of $F$ which switches the letters $x_i$ and $x_d$ and fixes the remaining ones.
For $i = 1,\ldots,d-1$, let $\iota_i \in \mathbb{Z}^{d-1}$ have $1$ in the $i$th coordinate and $0$ everywhere else. We may also view $\iota_i$ as an element of $V_{d-1}$ if convenient.

\begin{lemma} 
\label{factor}
Let $F$ be a free group of finite rank $d>1$ over $X=\{x_1,\ldots,x_d\}$ and let $G$ be a nonabelian $p$-hypersolvable group. 
Let $\varphi:F \to G$ be a surjective homomorphism. Then there exist some $t \in \{ 1,\ldots,d\}$, $C \leq \mathbb{Z}/(p-1)\mathbb{Z}$, $\alpha \in \{ 2,\ldots,p-1\}$ and surjective homomorphisms $\psi: F \to V_{d-1} \rtimes_{\alpha} C$ and $\theta: V_{d-1} \rtimes_{\alpha} C \to G$ such that the diagram
$$\xymatrix{
F \ar[rr]^{\varphi} \ar[d]_{\tau_t}&&G\\
F \ar[rr]_{\psi}&&V_{d-1} \rtimes_{\alpha} C \ar[u]_{\theta}
}$$
commutes and, for $i = 1,\ldots,d-1$, $\psi(x_i) = (\iota_i,c_i)$ for some $c_i \in C$.
\end{lemma}

\begin{proof}
By Lemma \ref{diameter}, $G$ must be of the form $V_{m} \rtimes_{\alpha} C$, with $m \leq d-1$ and $1\neq C \leq \mathrm{Aut}(\mathbb{Z}/p\mathbb{Z})\cong (\mathbb{Z}/p\mathbb{Z})^*$ acting diagonally (and nontrivially) on $V_{m}$. Let $\pi:G \to C$ be the canonical homomorphism. By Lemma \ref{primcyclic}, there exists some $\lambda \in \mathrm{Aut}(F)$ such that $\pi(\varphi(\lambda(x_1)))$ generates $C$. Write $x = \pi(\varphi(\lambda(x_1)))$. Then there exists some $\alpha \in \{ 2,\ldots,p-1\}$ such that $v^x = \alpha v$ for every $v \in V_m$. 

Write $G' = V_{d-1} \rtimes_{\alpha} C$. We build a surjective homomorphism $\theta': G' \to G$ by considering the projection $V_{d-1} \to V_{m}$ on the first $m$ coordinates and sending $x$ to $x$. The action is certainly preserved, hence $\theta'$ is a well-defined homomorphism. Naturally, we may view $G$ as a subgroup of $G'$ by identifying $V_m$ with the elements of $V_{d-1}$ with the last $d-1-m$ components equal to 0.

Since $\theta'$ is surjective, it follows from the universal property that there exists some homomorphism $\varphi':F \to G'$ such that $\theta' \circ \varphi' = \varphi$. The question is: can we make $\varphi'$ surjective? We prove it next.

For $i = 2,\ldots,d$, we have $\pi(\varphi(\lambda(x_i))) = x^{q_i}$ for some $0 \leq q_i \leq |C|-1$.  Let $\lambda':F \to F$ be the homomorphism defined by
$$\lambda'(x_i) = \left\{
\begin{array}{ll}
x_1&\mbox{ if }i = 1\\
x_ix_1^{-q_i}&\mbox{ if }1 < i \leq d
\end{array}
\right.$$
Since $\lambda'$ is surjective and $F$ is hopfian, then $\lambda' \in \mathrm{Aut}(F)$. Let $\eta = \varphi \circ \lambda \circ \lambda'$. Then there exist $u_1,\ldots,u_d \in V_m$ such that $\eta(x_1) = u_1x$ and $\eta(x_i) = u_i$ for $i = 2,\ldots,d$. Since $\eta$ is onto, it follows from condition (ii) of Lemma \ref{lem:genpsup} that 
$$\langle [u_1x,u_2],\ldots,[u_1x,u_d]\rangle = V_{m}.$$
Now 
$[u_1x,u_i] = (\alpha-1)u_i$ by Lemma \ref{lem:facts}(ii)
and so $V_m = \langle u_2,\ldots u_d\rangle$.

Since $V_{m}$ is a vector space of dimension $m$ over $\mathbb{Z}/p\mathbb{Z}$, the spanning subset $\{ u_2,\ldots,u_d\}$ must contain some basis. Hence there exists some $\lambda'' \in \mathrm{Aut}(F)$ induced by a permutation $\sigma \in S_d$ such that $\{ u_{\sigma(1)},\ldots,u_{\sigma(m)}\}$ is such a basis and $\sigma(d) = 1$. Let $\eta' = \eta \circ \lambda''$. We show that there exists some surjective homomorphism $\eta'':F \to G'$ such that $\theta' \circ \eta'' = \eta'$, i.e., making the diagram
$$\xymatrix{
F \ar[dr]^{\varphi} & \\
F \ar[u]^{\lambda\circ\lambda'} \ar[r]^{\eta} & G\\
F \ar[u]^{\lambda''} \ar[ur]^{\eta'} \ar[r]_{\eta''} & G' \ar[u]_{\theta'} 
}$$
commute.

Let $\eta'':F \to G'$ be the homomorphism defined by
$$\eta''(x_i) = \left\{
\begin{array}{ll}
\eta'(x_i)&\mbox{ if $1 \leq i \leq m$ or }i = d\\
\eta'(x_i) + \iota_i&\mbox{ if }m < i < d
\end{array}
\right.$$
It is routine to check that $\theta' \circ \eta'' = \eta'$ through the image of the letters. We show that $\eta''$ is onto.

For $i = 1,\ldots,m,d$, we have $\eta''(x_i) = \eta'(x_i) = \eta(\lambda''(x_i)) = \eta(x_{\sigma(i)})$, hence
$$\langle \eta''(x_1),\ldots,\eta''(x_m)\rangle = \langle \eta(x_{\sigma(1)}),\ldots,\eta(x_{\sigma(m)})\rangle = \langle u_{\sigma(1)},\ldots,u_{\sigma(m)}\rangle = V_m$$
and $\eta''(x_d) = \eta(x_{\sigma(d)}) = \eta(x_1) = u_1x$. It follows that $\iota_1,\ldots,\iota_m,x \in \mathrm{Im}\,\eta''$ and so $G= \mathrm{Im}\, \eta' \leq \mathrm{Im}\, \eta''$.  Now it follows that  $\iota_{m+1},\ldots,\iota_{d-1} \in \mathrm{Im}\,\eta''$ as well. Thus $\eta''$ is surjective.

Let $\varphi' = \eta'' \circ (\lambda \circ \lambda' \circ \lambda'')^{-1}$. Then $\varphi':F \to G'$ is a surjective homomorphism such that $\theta' \circ \varphi' = \varphi$.

For $i = 1,\ldots,d$, write $\varphi'(x_i) = v_ix^{c_i}$ with $v_i \in V_{d-1}$ and $c_i\in C$. By Lemmas \ref{lem:facts}(ii) and \ref{lem:genpsup}, $\{v_1,\ldots,v_d\}$ spans $V_{d-1}$.  
Since $V_{d-1}$ is a vector space of dimension $d-1$ over $\mathbb{Z}/p\mathbb{Z}$, the spanning subset $\{ v_1,\ldots,v_d\}$ must contain some basis. Hence there exists some $t \in \{1,\ldots,d\}$ such that 
$\{ v_{\tau_t(1)},\ldots, v_{\tau_t(d-1)} \}$ is a basis of $V_{d-1}$. 

Now we define a homomorphism $\zeta: G' \to G'$ by $\zeta(x) = x$ and $\zeta(\iota_i) = v_{\tau_t(i)}$ $(i = 1,\ldots,d-1)$. It is routine to check that $\zeta$ is a surjective homomorphism, hence an automorphism of $G'$. Finally, let $y_1,\ldots,y_{d-1} \in \mathbb{Z}/p\mathbb{Z}$ be such that $v_t = \sum_{i=1}^{d-1} y_iv_{\tau_t(i)}$. We define a homomorphism $\psi: F \to G'$ by
$$\psi(x_i) = \left\{
\begin{array}{ll}
\iota_ix^{c_{\tau_t(i)}}&\mbox{ if $i \in \{ 1,\ldots,d-1\}$}\\
(y_1,\ldots,y_{d-1})x^{c_t}&\mbox{ if $i = d$.}
\end{array}
\right.$$
We claim that $\zeta \circ \psi \circ \tau_t = \varphi'$. 
Indeed, if $i \neq t$ we get
$$\zeta(\psi(\tau_t(x_i))) = \zeta(\psi(x_{\tau_t(i)})) = \zeta(\iota_{\tau_t(i)}x^{c_i}) = v_ix^{c_i} = \varphi'(x_i).$$
as well as
$$\zeta(\psi(\tau_t(x_t))) = \zeta(\psi(x_d)) = \zeta((\sum_{i=1}^{d-1} y_i\iota_i)x^{c_t}) = (\sum_{i=1}^{d-1} y_iv_{\tau_t(i)})x^{c_t} = 
v_tx^{c_t} = \varphi'(x_t),$$
thus $\zeta \circ \psi \circ \tau_t = \varphi'$. Since $\zeta$ and $\tau_t$ are both automorphisms and $\varphi'$ is surjective, so is $\psi$.

Let $\theta = \theta' \circ \zeta$ so that we have the following commutative diagram
$$\xymatrix{
F \ar[rr]^{\varphi} \ar[dd]_{\tau_t} \ar[dr]_{\varphi'} && G\\
&G' \ar[ur]^{\theta'}&\\
F\ar[rr]_{\psi} &&G'\ar[ul]^{\zeta} \ar[uu]_{\theta}
}$$
Then 
$$\theta \circ \psi \circ \tau_t = \theta' \circ \zeta \circ \psi \circ \tau_t = \theta' \circ \varphi' = \varphi.$$
Since $\zeta \in \mathrm{Aut}\,(G')$ and $\theta'$ is surjective, so is $\theta$,
and the lemma is proved.
\end{proof}

\section{The proof of Theorem \ref{t:main}}\label{s:mainproof}

For an integer $d>1$ we consider the monoid ring
$$R = {\mathbb{Z}}[{\mathbb{Z}}^d \times{\mathbb{N}}^{d-1}] =
{\mathbb{Z}}[\alpha_1,\alpha_1^{-1} \hdots,\alpha_d,\alpha_d^{-1},\beta_1,\ldots,\beta_{d-1}],$$
which we choose to represent as some sort of polynomial ring over the indeterminates $\alpha_1,\ldots,\alpha_d$, $\beta_1,\ldots,\beta_{d-1}$, where the $\alpha_i$ admit inverses.

\begin{defn}
Let $p$ be a prime and $\mathcal{P} \subseteq R$. If there exists some nonzero ring homomorphism $\varphi:R \to \mathbb{Z}/p\mathbb{Z}$ such that
$\varphi(P) = 0$ for every $P \in \mathcal{P}$, we say that
$$(\varphi(\alpha_1), \varphi(\alpha_1^{-1}), \ldots, \varphi(\alpha_d), \varphi(\alpha_d^{-1}), \varphi(\beta_1),\ldots,\varphi(\beta_{d-1}))$$
is a common root modulo $p$ for the polynomials in $\mathcal{P}$.
\end{defn}

Let $T$ denote the additive group of $R$. We define an action of ${\mathbb{Z}}^{d}$ on $T^{d-1}$ by group automorphisms through
$$\begin{array}{rcl}
{\mathbb{Z}}^d \times T^{d-1}&\to&T^{d-1}\\
((k_1,\ldots,k_{d}),(P_1,\ldots,P_{d-1}))&\mapsto&(\alpha_1^{k_1}\ldots \alpha_d^{k_d}P_1,\ldots,\alpha_1^{k_1}\ldots \alpha_d^{k_d}P_{d-1})
\end{array}$$
Since the $\alpha_i$ are invertible, this is indeed a well-defined action, 
inducing a semidirect product $T^{d-1} \rtimes {\mathbb{Z}}^d$. 

Recall that $F$ is a free group of rank $d$ with basis $\{x_1,\ldots,x_d\}$. Also given a positive integer $k$ and an integer $1\leq i \leq k$, we let $\iota_i$ be the element of $\mathbb{Z}^k$ (or $T^k$) having $1$ in the $i$th coordinate and $0$ elsewhere. We define a homomorphism $\xi:F \to T^{d-1} \rtimes {\mathbb{Z}}^d$ by 
$$\xi(x_i) = \left\{
\begin{array}{ll}
(\iota_i,\iota_i)&\mbox{ if }1 \leq i < d\\
(\beta_1,\ldots,\beta_{d-1},\iota_d)&\mbox{ if } i = d.
\end{array}
\right.$$
It follows that
$$\xi(x_i^{-1}) = \left\{
\begin{array}{ll}
(-\alpha_i^{-1}\iota_i,-\iota_i)&\mbox{ if }1 \leq i < d\\
(-\alpha_d^{-1}\beta_1,\ldots,-\alpha_{d}^{-1}\beta_{d-1},-\iota_d)&\mbox{ if } i = d
\end{array}
\right.$$
thus we do not need to consider inverses for the $\beta_i$.
We write $\xi(w) = (\xi_1(w),\xi_2(w))$ for every $w \in F$.

For every prime $p > 2$, 
we define
$$\begin{array}{lll}
\mathcal{Q}(p) = \{ (c,\alpha,u,m)&|&c = (c_1,\ldots,c_d) \in  \{ 0,\ldots,p-2\}^d;\, 2 \leq \alpha,m \leq p-1;\, \alpha^{m} \equiv 1 \mod p;\\ &&\\
&&u = (u_1,\ldots,u_{d-1}) \in \{ 0,\ldots,p-1\}^{d-1} \}.\end{array}$$

Given $Q = (c,\alpha,u,m) \in \mathcal{Q}(p)$, 
let $\varphi_{p,Q}:R \to {\mathbb{Z}}/p{\mathbb{Z}}$ be the ring homomorphism extending the canonical homomorphism $\mathbb{Z} \to \mathbb{Z}/p\mathbb{Z}$ and such that
$$\varphi_{p,Q}(\alpha_i) = \alpha^{c_i},\quad \varphi_{p,Q}(\beta_j) = u_j\quad (1 \leq i \leq d,\; 1 \leq j \leq d-1).$$
Since $\alpha^{m} \equiv 1 \mod p$ and $\alpha\neq 1$, then $V_{d-1} \rtimes_{\alpha} \mathbb{Z}/m\mathbb{Z}$ is a well-defined $p$-hypersolvable group.

\begin{lemma}
Let $p$ be a prime and let $d>1$ be an integer. Given $Q=(c,\alpha,u,m)\in \mathcal{Q}(p)$, let $\Phi_{p,Q}:T^{d-1} \rtimes {\mathbb{Z}}^d \to V_{d-1} \rtimes_{\alpha} \mathbb{Z}/m\mathbb{Z}$ be defined by
$$\Phi_{p,Q}(P_1,\ldots,P_{d-1},k_1,\ldots,k_d) = (\varphi_{p,Q}(P_1),\ldots,\varphi_{p,Q}(P_{d-1}),\sum_{i=1}^d c_ik_i).$$
Then $\Phi_{p,Q}$ is a group homomorphism.
\end{lemma}

\begin{proof}
Let $Y,Y' \in T^{d-1} \rtimes {\mathbb{Z}}^d$, say $Y = (P_1,\ldots,P_{d-1},k_1,\ldots,k_d)$ and $Y' = (P'_1,\ldots,P'_{d-1},k'_1,\ldots,k'_d)$. Write $\alpha' = \alpha_1^{k_1}\ldots \alpha_d^{k_d}$. Then:
$$\begin{array}{rll}
\Phi_{p,Q}(YY')&=&\Phi_{p,Q}(P_1+\alpha' P'_1,\ldots,P_{d-1}+\alpha' P'_{d-1}, k_1+ k'_1,\ldots,k_d+k'_d)\\ &&\\
&=&(\varphi_{p,Q}(P_1+\alpha' P'_1),\ldots,\varphi_{p,Q}(P_{d-1}+\alpha' P'_{d-1}),\sum_{i=1}^d c_i(k_i+k'_i))\\ &&\\
\Phi_{p,Q}(Y)\Phi_{p,Q}(Y')&=&(\varphi_{p,Q}(P_1),\ldots,\varphi_{p,Q}(P_{d-1}),
\sum_{i=1}^d c_ik_i)\\
&&\hspace{.5cm}\cdot(\varphi_{p,Q}(P'_1),\ldots,\varphi_{p,Q}(P'_{d-1}),\sum_{i=1}^d c_ik'_i)\\ &&\\
&=&(\varphi_{p,Q}(P_1) + \alpha^{\sum_{i=1}^d c_ik_i}\varphi_{p,Q}(P'_1),\\
&&\hspace{.5cm}\ldots,
\varphi_{p,Q}(P_{d-1}) + \alpha^{\sum_{i=1}^d c_ik_i}\varphi_{p,Q}(P'_{d-1}),\sum_{i=1}^d c_i(k_i+k'_i)),
\end{array}$$
hence it suffices to show that
$$\varphi_{p,Q}(P_j+\alpha' P'_j) = \varphi_{p,Q}(P_j) + \alpha^{\sum_{i=1}^d c_ik_i}\varphi_{p,Q}(P'_j)$$
holds for $j = 1,\ldots,d-1$. Indeed, this follows easily from $\varphi_{p,Q}$ being a ring homomorphism and
$$\varphi_{p,Q}(\alpha') = \varphi_{p,Q}(\alpha_1^{k_1}\ldots \alpha_d^{k_d}) = \alpha^{c_1k_1}\ldots \alpha^{c_dk_d} =  \alpha^{\sum_{i=1}^d c_ik_i}.$$
Thus $\Phi_{p,Q}$ is a group homomorphism. 
\end{proof}

Now $\Phi_{p,Q}: T^{d-1} \rtimes {\mathbb{Z}}^d \to V_{d-1} \rtimes_{\alpha} \mathbb{Z}/m\mathbb{Z}$ induces a group homomorphism $\Phi'_{p,Q}:T^{d-1} \to V_{d-1}$ by
$$\Phi'_{p,Q}(P_1,\ldots,P_{d-1}) = \Phi_{p,Q}(P_1,\ldots,P_{d-1},0,\ldots,0).$$

If the rows of a matrix $M$ belong to $T^{d-1}$, then $\Phi'_{p,Q}(M)$ denotes the matrix obtained from $M$ by replacing each row by its image under $\Phi'_{p,Q}$. 

Suppose that $F$ is a free group of finite rank $d>1$ over $X=\{x_1,\ldots,x_d\}$ and $H$ is a finitely generated subgroup of $F$. Let $\{ w_1,\ldots,w_e\}$ be a basis of $H$. Then $\{ \tau_t(w_1),\ldots,\tau_t(w_e)\}$ is a basis of $\tau_t(H)$ for $t = 1,\ldots,d$.
We define an $\binom{e}{2} \times (d-1)$ matrix ${\mathcal{M}}(\tau_t(H))$ over $T$ by taking $\xi_1([\tau_t(w_r),\tau_t(w_s)])$
as row vectors for all $1 \leq r < s \leq e$. 
If $H = F$, we consider the basis $\{ x_1,\ldots,x_d\}$ of $F$, and define a  $\binom{d}{2} \times (d-1)$ matrix ${\mathcal{M}}(F)$ over $T$ by taking $\xi_1([x_r,x_s])$
as row vectors for all $1 \leq r < s \leq d$.

We  now introduce extra indeterminates $\gamma_1, \gamma_2, \gamma_3$. Write
$$R' = {\mathbb{Z}}[\alpha_1,\alpha_1^{-1}, \hdots,\alpha_d,\alpha_d^{-1},\beta_1,\ldots,\beta_{d-1},\gamma_1,\gamma_2,\gamma_3].$$
Let ${\mathcal{Y}}(F)$ denote the set of all $(d-1) \times (d-1)$ minors of ${\mathcal{M}(F)}$. For every $Y \in {\mathcal{Y}}(F)$, let 
$${\mathcal{P}}_Y = \{ \gamma_1\alpha_1\ldots \alpha_d\det(Y) - 1,\;
\alpha_1 + \ldots + \alpha_d - d - \gamma_2,\; \gamma_2\gamma_3 -1\} \subset R'.$$  
The purpose of $\gamma_1$ and $\gamma_3$ is to ensure that, for any common root of these polynomials, $\alpha_1,\ldots,\alpha_d,\det(Y)$ and $\gamma_2$ are all nonzero, and similarly the purpose of $\gamma_2$ is to ensure that, for any common root of these polynomials,  $\alpha_i \neq 1$ for some $i$.  

For $t = 1,\ldots,d$, let ${\mathcal{Y}}(\tau_t(H))$ denote the set of all $(d-1) \times (d-1)$ minors of ${\mathcal{M}}(\tau_t(H))$ and for $Y_0\in \mathcal{Y}(F)$, let 
$${\mathcal{P}}_{Y_0}(\tau_t(H)) = {\mathcal{P}}_{Y_0} \cup \{ \det(Y) \mid Y \in {\mathcal{Y}}(\tau_t(H)) \} \subset R'.$$  

\begin{lemma}
\label{roots}
Let $F$ be a free group of finite rank $d$ over $X=\{x_1,\ldots,x_d\}$. Let $H$ be a finitely generated subgroup of $F$. Suppose that $H$ is $\mathbf{Ab}$-dense.  Then the following conditions are equivalent:
\begin{itemize}
\item[(i)]
$H$ is not $\mathbf{Su}$-dense;
\item[(ii)] there exist some $p \in \mathbb{P}$, $t \in \{ 1,\ldots,d\}$ and $Y_0 \in {\mathcal{Y}(F)}$ 
 such that the polynomials in ${\mathcal{P}}_{Y_0}(\tau_t(H))$ admit a common root modulo $p$.
\end{itemize}
\end{lemma}

\begin{proof}
We are given a basis $W=\{w_1,\ldots,w_e\}$ of $H$ where  for $1\leq i \leq e$, $w_i$ is a reduced word over $X\cup X^{-1}$.
Note that if $d=1$ then $F\cong \mathbb{Z}$ is abelian and so $H$ is $\mathbf{Su}$-dense (as it is $\mathbf{Ab}$-dense). We therefore assume that $d>1$.
Since $H$ is $\mathbf{Ab}$-dense and $\mathbf{Ab}(2) \subset \mathbf{Ab}$, $H$ is $\mathbf{Ab}(2)$-dense by (\ref{e:cl}), and so $H$ satisfies condition $Q_2(H,F)$. 
(i) $\Rightarrow$ (ii). 
By Lemma \ref{lem:dense}(iii), $H$ fails condition $Q_p(H,F)$ for some prime $p$ which must be odd. Hence there exists some $N \unlhd F$ such that $HN < F$ and $F/N\cong \mathbb{Z}/p\mathbb{Z} \rtimes_{\alpha} C$ for some  $C\leq {\rm Aut}(\mathbb{Z}/p\mathbb{Z})\cong(\mathbb{Z}/p\mathbb{Z})^*$ and some $\alpha \in \{1,\dots,p-1\}$. Since $H$ is $\mathbf{Ab}$-dense, then $F/N$ must be nonabelian and so $C = \mathbb{Z}/m\mathbb{Z}$ for some $m>1$ dividing $p-1$ and $\alpha\in \{2,\dots,p-1\}$. Let $\varphi:F \to F/N$ be the canonical homomorphism. 
By Lemma \ref{factor}, there exist some $t \in \{ 1,\ldots,d\}$, and surjective homomorphisms $\psi: F \to V_{d-1} \rtimes_{\alpha} C$ and $\theta: V_{d-1} \rtimes_{\alpha} C \to F/N$ such that $\theta \circ \psi \circ \tau_t = \varphi$ and, for $i = 1,\ldots,d-1$, $\psi(x_i) = (\iota_i,c_i)$ for some $c_i \in C$.
Write $\psi(x_d) = (u_1,\ldots,u_{d-1},c_d)$, $c = (c_1,\ldots,c_d)$ and $u = (u_1,\ldots,u_{d-1})$. Since $\alpha^m \equiv 1 \mod p$ is a requirement for the action being well defined, we get 
$Q = (c,\alpha,u,m) \in \mathcal{Q}(p)$.

Consider  now the ring homomorphism $\varphi_{p,Q}:R \to \mathbb{Z}/p\mathbb{Z}$. We use this homomorphism to find a common root modulo $p$ for all the polynomials in ${\mathcal{P}}_{Y_0}(\tau_t(H))$ for some $Y_0 \in {\mathcal{Y}}(F)$. We do not need to worry about the auxiliary indeterminates $\gamma_i$ if we understand their purpose. That is, it suffices to show that:
\begin{itemize}
\item[(A1)]
$\varphi_{p,Q}(\det(Y_0)) \neq 0$ for some $Y_0 \in {\mathcal{Y}}(F)$;
\item[(A2)]
$\varphi_{p,Q}(\alpha_i) \neq 0$ for $i = 1,\ldots,d$;
\item[(A3)]
$\varphi_{p,Q}(\alpha_i) \neq 1$ for some $1 \leq i \leq d$;
\item[(A4)]
$\varphi_{p,Q}(\det(Y)) = 0$ for every $Y \in {\mathcal{Y}}(\tau_t(H))$.
\end{itemize}

Recall now $\Phi_{p,Q}$, $\Phi_{p,Q}'$, $\xi$ and $\mathcal{M}:=\mathcal{M}(F)$. We consider  the composition 
$$\Phi_{p,Q} \circ \xi: F \to V_{d-1} \rtimes_{\alpha} \mathbb{Z}/m\mathbb{Z} = V_{d-1} \rtimes_{\alpha} C.$$
It follows from the definitions that $\Phi_{p,Q}(\xi(x_i)) = \psi(x_i)$ for $i = 1,\ldots,d$, hence $\Phi_{p,Q} \circ \xi = \psi$.

Since $\psi$ is surjective, it is easy to see that $C = \langle c_1,\ldots,c_d\rangle$ and Lemma \ref{lem:genpsup} yields
$$\langle [\Phi_{p,Q}(\xi(x_i)),\Phi_{p,Q}(\xi(x_j))] \mid 1 \leq i < j \leq d\rangle = V_{d-1},$$
hence
$$\langle \Phi'_{p,Q}(\xi_1([x_i,x_j])) \mid 1 \leq i < j \leq d\rangle = V_{d-1}.$$
This is equivalent to saying that the row vectors of $\Phi'_{p,Q}({\mathcal{M}})$ generate $V_{d-1}$, hence this matrix has necessarily rank $d-1$.  If
$${\mathcal{M}} = 
\begin{pmatrix}
P_{1,1}&\ldots&P_{1,d-1}\\
\ldots&\ldots&\ldots\\
P_{m,1}&\ldots&P_{m,d-1}
\end{pmatrix}$$
then
$$\Phi'_{p,Q}({\mathcal{M}}) = 
\begin{pmatrix}
\varphi_{p,Q}(P_{1,1})&\ldots&\varphi_{p,Q}(P_{1,d-1})\\
\ldots&\ldots&\ldots\\
\varphi_{p,Q}(P_{m,1})&\ldots&\varphi_{p,Q}(P_{m,d-1})
\end{pmatrix}$$
Since $\Phi'_{p,Q}({\mathcal{M}})$ has rank $d-1$ and $\varphi_{p,Q}$ is a ring homomorphism, there exists some 
$Y_0 \in \mathcal{Y}(F)$ such that $\varphi_{p,Q}(\det(Y_0)) \neq 0$. Thus (A1) holds.

On one hand, $\varphi_{p,Q}(\alpha_i) = \alpha^{c_i} \neq 0$ since $2 \leq \alpha \leq p-1$, hence (A2) holds. On the other hand, $C = \langle c_1,\ldots,c_d\rangle$ nontrivial implies $c_i \neq 0$ for some $i$, therefore (A3) holds as well.

Finally, let $Y \in {\mathcal{Y}}(\tau_t(H))$. Suppose that $\varphi_{p,Q}(\det(Y)) \neq 0$. The same argument used above shows that the row vectors of $\Phi'_{p,Q}({\mathcal{M}}(\tau_t(H)))$ generate $V_{d-1}$. Since $\Phi_{p,Q} \circ \xi = \psi$, we get
$$\langle [\psi(\tau_t(w_i)),\psi(\tau_t(w_j))] \mid 1 \leq i < j \leq e\rangle = V_{d-1},$$
hence $V_{d-1} \leq \psi(\tau_t(H))$.
On the other hand, since $\tau_t(H)$ is $\mathbf{Ab}$-dense by Lemma \ref{autospres}, we can compose $\psi$ with the canonical homomorphism $\pi: V_{d-1} \rtimes_{\alpha} C \to C$ to derive $$\tau_t(H)\mathrm{ker}(\pi\circ \psi)=F$$ and consequently $\pi(\psi(\tau_t(H))) = C$. Together with $V_{d-1} \leq \psi(\tau_t(H))$, this yields 
$\psi(\tau_t(H)) = V_{d-1} \rtimes_{\alpha} C$ and so $(\tau_t(H))\mathrm{Ker}\,\psi = F$.

Let $g \in F$. Then $\tau_t(g) = (\tau_t(h))z$ for some $h \in H$ and $z \in \mathrm{Ker}\,\psi$. Hence
$$g = \tau_t^2(g) = h\tau_t(z).$$
Since $\psi(\tau_t(\tau_t(z))) = \psi(z) = 1$, we get $H\mathrm{Ker}(\psi\circ \tau_t) = F$. Since $\theta \circ \psi \circ \tau_t = \varphi$, then $\mathrm{Ker}(\psi\circ \tau_t) \leq\mathrm{Ker}\,\varphi = N$ and so $HN = F$, a contradiction. Therefore (A4) holds and so does condition (ii).

(ii) $\Rightarrow$ (i). 
Assume that condition (ii) holds for $p \in \mathbb{P}$, $Y_0 \in {\mathcal{Y}}(F)$ and $t \in \{ 1,\ldots,d\}$. Let $a_1,\ldots,a_{d},u_1,\ldots,u_{d-1},g_1,g_2,g_3$ be a common root modulo $p$ for the polynomials in ${\mathcal{P}}_{Y_0}(\tau_t(H))$. 

Since $a_i \not\equiv 0 \mod p$ for $i = 1,\ldots,d$, we can write $C = \langle a_1,\ldots, a_{d} \rangle \leq (\mathbb{Z}/p\mathbb{Z})^*$. Note that $a_i \not\equiv 1 \mod p$ for some $i$, hence $|C| > 1$.
Let $x$ denote a generator of $C$ and let $\alpha \in \{ 2,\ldots,p-1\}$ represent $x$ modulo $p$. For $i = 1,\ldots,d$, we have $a_i = \alpha^{c_i}$ for some $0 \leq c_i \leq |C|-1$. Write $c = (c_1,\ldots,c_d)$ and $u = (u_1,\ldots,u_{d-1})$. Since $\alpha^{|C|} \equiv 1 \mod p$, we can define $Q = (c,\alpha,u,|C|) \in \mathcal{Q}(p)$.

We consider now the $p$-hypersolvable group $G = V_{d-1} \rtimes_{\alpha} C$, viewing $C$ as $\mathbb{Z}/|C|\mathbb{Z}$ (hence $C = \langle c_1,\ldots,c_d\rangle$). 
Let $\psi:F \to G$ be the homomorphism defined by
$$\psi(x_i) = \left\{
\begin{array}{ll}
(\iota_i,c_i)&\mbox{ if }1 \leq i < d\\
(u_1,\ldots,u_{d-1},c_d)&\mbox{ if }i = d.
\end{array}
\right.$$ 
Once again, we have $\Phi_{p,Q} \circ \xi = \psi$. We prove that $\psi$ is surjective, by showing that $\psi(x_1),\ldots, \psi(x_d)$ satisfy the conditions of Lemma \ref{lem:genpsup}.

On the one hand, 
$$\mathrm{gcd}(|C|, c_1,\ldots,c_d) = 1$$
follows from $\langle c_1,\ldots,c_d \rangle = C = \mathbb{Z}/|C|\mathbb{Z}$, and so condition (i) of Lemma \ref{lem:genpsup} is satisfied.

On the other hand,
$$[\psi(x_i),\psi(x_j)] = [\Phi_{p,Q}(\xi(x_i)),\Phi_{p,Q}(\xi(x_j))] = \Phi'_{p,Q}(\xi_1([x_i,x_j]))$$
holds for all $1 \leq i < j \leq d$, and once again we are applying the ring homomorphism $\varphi_{p,Q}$ to each of the entries of the matrix $\mathcal{M}(F)$. By our assumption, the corresponding equality of type (A1) is satisfied, that is $\varphi_{p,Q}(\det(Y_0)) \neq 0$, which implies that the matrix with rows $[\psi(x_i),\psi(x_j)]$ has rank $d-1$. Therefore condition (ii) of Lemma \ref{lem:genpsup} is satisfied and so $\psi$ is surjective. 

Let $N = \mathrm{Ker}\,\psi \unlhd F$. Next we prove that $\tau_t(H)N < F$, which implies that $\tau_t(H)$ is not $\mathbf{Su}$-dense. By Lemma \ref{autospres}, $H$ is not $\mathbf{Su}$-dense either.

Indeed, suppose that $\tau_t(H)N = F$. Then $\psi(\tau_t(H)) = G$ and so Lemma \ref{lem:genpsup} implies that
$$\langle  [\psi(\tau_t(w_i)),\psi(\tau_t(w_j))] \mid 1 \leq i < j \leq e\rangle = V_{d-1},$$
yielding
$$\langle  [\Phi_{p,Q}(\xi(\tau_t(w_i))),\Phi_{p,Q}(\xi(\tau_t(w_j)))] \mid 1 \leq i < j \leq e\rangle = V_{d-1}$$
and consequently
$$\langle  \Phi'_{p,Q}(\xi_1([\tau_t(w_i),\tau_t(w_j)])) \mid 1 \leq i < j \leq e\rangle = V_{d-1}.$$
It follows that the matrix obtained by applying $\varphi_{p,Q}$ to every entry of ${\mathcal{M}}(\tau_t(H))$ has rank $d-1$, contradicting condition (A4), which follows from our assumption. Therefore $\tau_t(H)N < F$ and so $H$ is not $\mathbf{Su}$-dense.
\end{proof}

We can now prove Theorem \ref{t:main}.

\textit{Proof of Theorem \ref{t:main}.}
Since $\mathbf{Ab} \subset \mathbf{Su}$, it follows from (\ref{e:cl}) that being $\mathbf{Ab}$-dense is a necessary condition for being $\mathbf{Su}$-dense. Since $\mathbf{Ab}$-denseness is decidable by \cite[Theorem 4.4]{MSTmeta}, we may assume that $H$ is $\mathbf{Ab}$-dense. Also by \cite[Corollary 2.4]{MSTcyclic}, if $F$ is not of finite rank then $H$ is not $\mathbf{Su}$-dense. We can therefore assume that $F$ is of finite rank $d$. We assume that $d>1$, as if $d=1$ then $F\cong \mathbb{Z}$ is abelian and so $H$ is $\mathbf{Su}$-dense (as it is $\mathbf{Ab}$-dense).
By Lemma \ref{roots} $H$ is not $\mathbf{Su}$-dense if and only if  there exist some $p \in \mathbb{P}$, $t \in \{ 1,\ldots,d\}$ and $Y_0 \in {\mathcal{Y}(F)}$  such that the polynomials in ${\mathcal{P}}_{Y_0}(\tau_t(H))$ admit a common root modulo $p$.

We fix $t \in \{ 1,\ldots,d\}$ and $Y_0 \in {\mathcal{Y}}(F)$. To complete the proof of the theorem, it suffices to show that it is decidable whether or not there is a prime $p$ such that the polynomials in ${\mathcal{P}}_{Y_0}(\tau_t(H))$ admit a common root modulo $p$.
The first thing we need to do is to replace ${\mathcal{P}}_{Y_0}(\tau_t(H)) \subset R'$ by some equivalent subset of polynomials 
$${\mathcal{P}}'_{Y_0}(\tau_t(H)) \subset R'' = {\mathbb{Z}}[\alpha_1, \hdots,\alpha_d,\beta_1,\ldots,\beta_{d-1},\gamma_1,\gamma_2,\gamma_3].$$

Indeed, there exists some $N \geq 1$ such that the matrices 
$$\alpha_1^N\ldots \alpha_d^N\mathcal{M}(F) \quad \textrm{and} \quad \alpha_1^N\ldots \alpha_d^N\mathcal{M}(\tau_t(H))$$
have all the entries in $R''$. Let 
$$\begin{array}{lll}
\mathcal{P}'_{Y_0}(\tau_t(H))&=&\{\gamma_1\alpha_1^N\ldots \alpha_d^N \det(Y_0)-1, \alpha_1 + \ldots +\alpha_d -d-\gamma_2, \gamma_2\gamma_3 -1\}\\ &&\\
&\cup&\{ \det(Y) \mid Y \in \alpha_1^N\ldots \alpha_d^N \mathcal{Y}(\tau_t(H)) \} \subset R''.
\end{array}$$
Any common root modulo $p$ for the polynomials in ${\mathcal{P}}_{Y_0}(\tau_t(H))$  induces by restriction  
a common  root modulo $p$ for the polynomials in ${\mathcal{P}}'_{Y_0}(\tau_t(H))$ (adapting the value of $\gamma_1$, which unique purpose is to force $\alpha_1,\ldots,\alpha_d,\mathrm{det}(Y)$ to assume nonzero values). The converse holds because all the $\alpha_i$ are forced to assume nonzero values, so all we need is to interpret $\alpha_i^{-1}$ as the inverse of $\alpha_i$ (resetting appropriately the value of $\gamma_1$).
It now suffices to show that it is decidable whether or not
$$S({\mathcal{P}}'_{Y_0}(\tau_t(H))) \neq \emptyset.$$
Since ${\mathcal{P}}'_{Y_0}(\tau_t(H))$ is a finite subset of $R''$, it follows from Theorem \ref{t:jar} that there exists some  $f \in \mathbb{Z}[\alpha]$ such that $S({\mathcal{P}}'_{Y_0}(\tau_t(H))) = S(f)$. Since $S(f) = \emptyset$ if and only if $f = \pm 1$, Corollary \ref{c:sys} yields the result. $\square$

Thus we have answered positively Question 1 for the pseudovariety $\mathbf{V}=\mathbf{Su}$, however Questions 2-5 remain open at this stage (for $\mathbf{V}=\mathbf{Su}$). 

\section*{Acknowledgements}

The first author acknowledges support from the Centre of Mathematics of the University of Porto, which is financed by national funds through the Funda\c c\~ao para a Ci\^encia e a Tecnologia, I.P., under the project with references UIDB/00144/2020 and  UIDP/00144/2020.

The second author acknowledges support from the Centre of Mathematics of the University of Porto, which is financed by national funds through the Funda\c c\~ao para a Ci\^encia e a Tecnologia, I.P., under the project with reference UIDB/00144/2020. 

The third author was supported by the Engineering and Physical Sciences Research Council, grant number 
EP/T017619/1.

\vspace{1cm}

{\sc Claude Marion, Centro de
Matem\'{a}tica, Faculdade de Ci\^{e}ncias, Universidade do
Porto, R. Campo Alegre 687, 4169-007 Porto, Portugal}

{\em E-mail address}: claude.marion@fc.up.pt

\bigskip

{\sc Pedro V. Silva, Centro de
Matem\'{a}tica, Faculdade de Ci\^{e}ncias, Universidade do
Porto, R. Campo Alegre 687, 4169-007 Porto, Portugal}

{\em E-mail address}: pvsilva@fc.up.pt

\bigskip

{\sc Gareth Tracey, Mathematics Institute, University of Warwick, Coventry CV4 7AL, U.K.}

{\em E-mail address}: Gareth.Tracey@warwick.ac.uk


\begin{thebibliography}{99}

\bibitem{Baer} R. Baer. Classes of finite groups and their properties. Illinois J. Math. \textbf{1} (1957), 115--187.
 
\bibitem{Buchberger} B. Buchberger. An algorithm for finding the basis elements of the residue class ring of a zero dimensional polynomial ideal. Translated from the 1965 German original by Michael P. Abramson. J. Symbolic Comput. \textbf{41} (2006), 475--511.
 
 \bibitem{vDries} L. van den Dries. A remark on Ax's theorem on solvability modulo primes. Math. Z. \textbf{208} (1991), 65--70.
 
  \bibitem{Gaschutz} W. Gasch\"{u}tz. Zu einem von B.H. und H. Neumann gestellten Problem. Math. Nachr. \textbf{14} (1956),
249--252.  

\bibitem{Hall49} M. Hall. Coset representations in free groups. Trans. Amer. Math. Soc. \textbf{67} (1949), 421--432.

\bibitem{Hall50} M. Hall. A topology for free groups and related groups. Annals of Mathematics \textbf{52} (1950), 127--139.

\bibitem{HL} B. Herwig and D. Lascar. Extending partial automorphisms and the profinite topology on free groups.  Trans. Amer. Math. Soc. \textbf{352} (2000), 1985--2021.


\bibitem{Jarviniemi} O. Jarviniemi. Solvability of a system of polynomial equations modulo primes.  Bull. Aust. Math. Soc. 106 (2022), 404--407.
 
 \bibitem{MSW} S. Margolis, M. Sapir and  P. Weil. Closed subgroups in pro-$V$ topologies and the extension problem for inverse automata. Internat. J. Algebra Comput. \textbf{11} (2001), 405--445.
 

 \bibitem{MSTmeta} C. Marion, P. V. Silva and G. Tracey. The pro-$k$-solvable topology on a free group, arXiv:2304.10235. Preprint, 2023. 

\bibitem{MSTcyclic}  C. Marion, P. V. Silva and G. Tracey. On the closure of cyclic subgroups of a free group in pro-$\mathbf{V}$ topologies, arXiv:2304.10230. Preprint, 2023.
 
 \bibitem{MW} D. Masser, G. W\"{u}stholz. Fields of Large Transcendence Degree Generated by Values of Elliptic Functions. Invent. Math. \textbf{72} (1983), 407--464. 
  
\bibitem{RZ} L. Ribes and P. Zalesskii. The pro-$p$ topology of a free group and algorithmic problems in semigroups. Internat. J. Algebra Comput. \textbf{4} (1994), 359–374.



\end{thebibliography}
\end{document}